\documentclass[11pt,reqno,oneside]{amsart}
\pdfoutput=1

\usepackage{amsmath, amsthm, amssymb}
\usepackage[applemac]{inputenc}

\usepackage{tensor}
\usepackage{bbm}
\usepackage{geometry}
\usepackage{soul}
\usepackage[pdftex]{graphicx}
\usepackage{float}
\usepackage{subfig}
\usepackage[english]{babel}
\usepackage{tikz} 
\usepackage{setspace}
\usetikzlibrary{arrows,decorations.pathreplacing,decorations.markings,shapes.geometric}
\tikzset{
    bt/.style={draw=blue,thick},
    ns/.style={circle,draw=blue,fill=blue, inner sep=0pt, minimum size=2mm},
    string/.style={draw=#1, postaction={decorate}, decoration={markings,mark=at position .45 with {\arrow[blue]{triangle 60}}}},
    doublestring/.style={draw=#1, postaction={decorate}, decoration={markings, mark=at position .7 with {\arrow[blue]{triangle 60}}, 
    mark=at position .3 with {\arrowreversed[blue]{triangle 60}}}},
    costring/.style={draw=#1, postaction={decorate}, decoration={markings,mark=at position .55 with {\arrow[draw=#1]{<}}}},
    arr/.style={string=blue, thick},
    doublearr/.style={doublestring=blue, thick},
    lin/.style={blue},
    dlin/.style = {blue, dashed, thick},
    dot/.style={circle,draw=#1,fill=#1,inner sep=1pt},
}

\usepackage{dsfont}
\usepackage{color}
\usepackage[color,matrix,arrow]{xy}
\graphicspath{{graphics/}}
\numberwithin{equation}{section}
\usepackage{fancyhdr}
\theoremstyle{definition}

\newtheorem{Thm}{Theorem}[section]
\newtheorem{Cor}{Corollary}[section]

\newtheorem{Lemma}{Lemma}[section]
\newtheorem{Rem}{Remark}[section]

\newcommand{\sh}[1]{\mathcal{#1}}

\DeclareMathOperator{\id}{id}

\DeclareMathOperator{\SL}{SL}

\DeclareMathOperator{\C}{\mathbb{C}}

\DeclareMathOperator{\Z}{\mathbb{Z}}

\DeclareMathOperator{\Proj}{\mathbb{P}}

\DeclareMathOperator{\Gm}{\mathbb{G}_m}

\DeclareMathOperator{\Dual}{\mathbb{D}}

\DeclareMathOperator{\Aff}{\mathbb{A}}

\DeclareMathOperator{\central}{cent}
\DeclareMathOperator{\can}{can}
\DeclareMathOperator{\inc}{inc}

\usepackage{hyperref}
\hypersetup{colorlinks,allcolors=black}

\title{Twisted $\sh{D}$-module extensions of local systems on a certain subvariety isomorphic to $\Gm^2$ of the affine flag variety of $\SL_2$}

\begin{document}

\author{Claude Eicher}
\address{Skolkovo Institute of Science and Technology, Moscow, Russia}
\email{C.Eicher@skoltech.ru}
\date{\today}
\setcounter{tocdepth}{3}
\maketitle

\setstretch{1.4}

\begin{abstract}
We introduce a family of rank-one local systems in the category of twisted $\sh{D}$-modules on a certain subvariety isomorphic to $\Gm^2$
of the affine flag variety of $\SL_2$. We then give a criterion for these local systems, in terms of their parameters, to extend cleanly
in the sense of $\sh{D}$-modules. 
\end{abstract}

\tableofcontents

\section{Introduction}
\subsection{Overview}
In the present work we prove a basic result concerning a $\sh{D}$-module construction based on a subvariety $O$ of the affine flag variety
of $\SL_2$ over $\C$, isomorphic to $\Gm^2$ and contained in a two dimensional Schubert cell. Namely, we construct a family of rank-one local systems in the category of arbitrarily complex
twisted right $\sh{D}$-modules on $O$. The twist is w.r.t. the usual torsor over the affine flag variety. The family is naturally parametrized by two complex numbers describing the monodromies of the local system, henceforth called monodromy parameters, and its construction depends on a choice of coordinates in $O$.\par
Our result, 
Theorem \autoref{Thm:cleanness}, states a criterion for the $\sh{D}$-module extension of such a local system to the affine flag variety to be clean: certain linear combinations of the twist and monodromy parameters should be non-integral. Its proof consists of a straightforward computation in 
coordinates on an open cover, by four affine planes, of the Zariski closure of $O$.\par
The result only depends on the embedding
of $O$ into its Zariski closure, or more precisely on the embedding of the corresponding total spaces of the above mentioned torsor, and not on the
 embedding of $O$ into the flag variety itself. 
A reason we nevertheless use the affine flag variety to formulate our result is that we would like to indicate an implication for 
the representation of the affine Kac-Moody algebra $\widehat{\mathfrak{sl}_2}$ in the space of global sections of the
twisted $\sh{D}$-module extensions under consideration. And in fact we expect that the structure of $\widehat{\mathfrak{sl}_2}$-representations of the
kind of these global sections, which is at present poorly understood, can be elucidated further using $\sh{D}$-module 
 techniques. \par

\subsection{Context} The subvariety $O$ can be compared with the
subvariety $X_{s_i} \cap s_i X_{s_i}$, isomorphic to $\Gm$, of the affine flag variety of any almost simple and simply connected linear
algebraic group over $\C$. Here $X_{s_i}$ denotes the (one dimensional) Schubert cell associated with the simple reflection $s_i$ of the affine Weyl
group. In \cite{Eic16b}, \cite{Eic20} we consider $\sh{D}$-module constructions based on a
 Kummer local system on $X_{s_i} \cap s_i X_{s_i}$, in the case of
integral and arbitrary complex twist, respectively. The twisted $\sh{D}$-module extensions of \cite{Eic20} analogous to the
ones considered presently are called $\sh{R}_{?s_i,i,\lambda, \mu}$ there. We will not recall the definition of these twisted $\sh{D}$-modules
on the affine flag variety, but content ourselves
with stating that they are constructed in a  natural way from the choice $? \in \{!,*\}$ of $!$- or $*$-extension, a simple affine root $i$,
an arbitrary complex affine weight $\lambda$ describing the twist, and the complex number $\mu$ describing the monodromy of the Kummer local system. The analogue of Theorem \autoref{Thm:cleanness} states that the
canonical morphism in the sense of \cite{Ber} $\sh{R}_{!s_i,i,\lambda,\mu} \rightarrow \sh{R}_{*s_i,i,\lambda,\mu}$ is an isomorphism if and only if $\mu \notin \Z$
and $\lambda(h_i)-\mu \notin \Z$. Here $h_i$ denotes the coroot.\par

At the level of orbit stratifications, the situation is as follows. The group scheme  whose $R$-valued points are
\begin{align}\label{eq:groupSL2}
\left\{\left. g = \begin{pmatrix} A_0 + t^2a & t^2 b \\ t^2 c & D_0+t^2d\end{pmatrix}\; \right \vert A_0, D_0 \in R,\; a,b,c,d \in R[[t]],\; \det g = 1\right\} \rtimes R^{\times}
\end{align}
acts in a natural way on the affine flag variety of $\SL_2$ and plays the role the group scheme $I \cap {^{s_i} I}$
does in \cite{Eic16b}, \cite{Eic20}. Here $R^{\times}$ acts by loop rotations, i.e. by the automorphism induced by $t \mapsto rt$, $I$ denotes an Iwahori group
of the loop group, and $^{s_i}I$ its $s_i$-conjugate. The subvariety $O$ is an orbit for this action and any other orbit of dimension less or equal
than two arises as an orbit for the action of a group scheme strictly containing the one defined by \eqref{eq:groupSL2}. In this sense,
$O$ can be considered as the basic orbit for this group action.
In the same way, the subvariety $X_{s_i} \cap s_i X_{s_i}$ can be considered as the basic orbit for the action of the group $I \cap {^{s_i} I}$. 

\section*{Acknowledgements}
We would like to thank B.L. Feigin and G. Felder for discussions related to this article. 

\section{Notation}
When it is clear from the indication of its domain and codomain, we sometimes, for simplicity, omit from notation that we understand a restriction 
of the morphism under consideration. Any inclusion of a subvariety is denoted by $\inc$. The sheaf-theoretic direct image w.r.t. a morphism $f$ is denoted by $f_{\cdot}$. When $R$ is a commutative $\C$-algebra, we denote by $R^{\times}$ its group of invertible elements. 
The restriction of a line bundle or torsor $\sh{F}$ to a subvariety $S$ is denoted by $\sh{F} \vert S$. 

\section{Setup}
Let $\SL_2((t))$ be the algebraic loop group of $\SL_2$ over $\C$. 
Let $X = \SL_2((t))/I$ be the affine flag variety of $\SL_2$, an ind-projective ind-variety over $\C$.
Let $\widetilde{\pi}: \widetilde{X} = \SL_2((t))/I^u \rightarrow \SL_2((t))/I$ be the canonical projection, a $T^{\circ}$-torsor.  Here 
$I$ is the standard Iwahori group scheme of $\SL_2((t))$ with $R$-valued points
\begin{align}\label{eq:RpointsofI}
I(R) = \left\{\left. g=\begin{pmatrix} a & b \\ tc & d\end{pmatrix}\; \right \vert \; a,b,c,d \in R[[t]],\; \det g = 1\right\}\;,
\end{align}
$I^u$ is the pro-unipotent radical of $I$, and $T^{\circ} \subseteq \SL_2$ is the torus of diagonal matrices.  Let $\sh{C}$ be the level line bundle on $\SL_2((t))/I$, see e.g. \cite{Zhu10}. It is normalized such that $\sh{C} \vert (P_i/I) \cong \sh{O}_{\Proj^1}(1)$
for each $i \in \{0,1\}$, where $\SL_2((t)) \supseteq P_i \supseteq I$ is the parabolic subgroup associated to $i$.  
We denote $\widetilde{\sh{C}} = \widetilde{\pi}^* \sh{C}$, a line bundle on $\widetilde{X}$. For any subvariety $S \subseteq X$ we
denote $\widetilde{S} = \widetilde{\pi}^{-1}(S) \subseteq \widetilde{X}$. 
Let $(\cdot)^{\times}$ denote the $\Gm$-torsor of invertible sections of a line bundle.
We may, and will, in order to distinguish it from other $\Gm$ appearing in the text, denote the structure group $\Gm$ of 
$\sh{C}^{\times}$ and $\widetilde{\sh{C}}^{\times}$ by $\Gm^{\central}$. Here $^{\central}$ stands for ``central''
because of the possible interpretation of $\Gm^{\central}$ in terms of the central extension of $\SL_2((t))$, which, however, will
not play an important role in this article. 
Let $X_w$ be the $I$-orbit in $X$, i.e. the finite dimensional Schubert cell, associated to the element $w$
of the affine Weyl group $W$ of $\SL_2$
and $\overline{X_w}$ its Zariski closure in $X$, the corresponding Schubert variety. Let $s_i$, $i \in \{0,1\}$, denote the
two simple reflections in $W$.
 We recall that the Demazure resolution $(P_1 \times_I P_0)/I \rightarrow \overline{X_{s_1s_0}}$ is an isomorphism,
 in particular $\overline{X_{s_1s_0}}$ is a $\Proj^1$-bundle over $\Proj^1$.

\section{Local coordinates and trivializations}
\subsection{Local coordinates on $\overline{O}$}
We set $O(R) = \begin{pmatrix} t & R^{\times} t^{-1}+R^{\times} \\ 0 & t^{-1}\end{pmatrix}I(R)$, then 
$O(R)$ are the $R$-valued points of a locally closed subvariety $O$ of $X$.  

\begin{Lemma} \label{Lemma:Phi_i}
There are open subsets $\{U_i\}_{i=1,2,3,4}$ of $\overline{X_{s_1s_0}}$ such that
the isomorphism $O \xrightarrow{\cong} \Gm^2$ given by
\begin{align*}
\begin{pmatrix} t  & a_{-1}t^{-1}+a_0 \\ 0 & t^{-1}\end{pmatrix}I \mapsto &(a_{-1},a_0) \\
\mapsto &\left(\frac{1}{a_{-1}}, \frac{a_0}{a_{-1}^2}\right) \\
\mapsto &\left(a_{-1}, \frac{1}{a_0}\right) \\
\mapsto &\left(\frac{1}{a_{-1}}, \frac{a_{-1}^2}{a_0}\right)
\end{align*}
extends to an isomorphism $U_i \xrightarrow{\cong} \Aff^2$ for $i =1,2,3,4$, respectively. We
denote the inverse of this isomorphism by $\Phi_i: \Aff^2 \xrightarrow{\cong} U_i$. 
\end{Lemma}

\begin{proof}
That the isomorphisms extend follows from the fact that
\begin{align*}
\begin{pmatrix} t & a_{-1}t^{-1}+a_0 \\ 0 & t^{-1}\end{pmatrix} I &\to \begin{pmatrix} 1 & -\frac{a_{-1}^2}{a_0}t^{-1} \\ 0 & 1\end{pmatrix}I\qquad
a_{-1} \to \infty,\; \frac{a_0}{a_{-1}^2}\; \text{fixed}\\
&\to \begin{pmatrix}  -a_{-1} & 1 \\ -1 & 0\end{pmatrix}I \qquad a_{-1}\; \text{fixed},\;  a_0 \to \infty
\end{align*}
and that

\begin{align*}
\begin{pmatrix} 1 & a_{-1}t^{-1} \\ 0 & 1\end{pmatrix} I &\to \begin{pmatrix} 0 & t^{-1} \\ -t & 0\end{pmatrix}I\qquad a_{-1} \to \infty \\
\begin{pmatrix} t & a_{-1}t^{-1} \\ 0 & t^{-1}\end{pmatrix}I &\to \begin{pmatrix} 0 & t^{-1} \\ -t & 0\end{pmatrix}I \qquad a_{-1} \to \infty \\
\begin{pmatrix}  t & a_0 \\  0 & t^{-1}\end{pmatrix}I &\to \begin{pmatrix} 0 & 1 \\ -1 & 0\end{pmatrix}I \qquad a_0 \to \infty\\
\begin{pmatrix} a_0 & 1 \\ -1 & 0\end{pmatrix}I &\to I\qquad a_0 \to \infty\;,
\end{align*}
which is easily shown by multiplying from the right by a suitable element
of $I$. 
\end{proof}

For any $i$, we denote by $(x,y)$ the coordinates of the source of $\Phi_i$.  
We have $U_1 = X_{s_1s_0}$, hence $\overline{O} = \overline{X_{s_1s_0}}$, and  $\overline{O} = \bigcup_{i=1}^4 U_i$, i.e. $\{U_i\}_{i=1,2,3, 4}$ is an open cover of $\overline{O}$. 

\subsection{Local trivializations of $\widetilde{\pi} \vert \overline{O}$}

In this section we write down a trivialization of the $T^{\circ}$-torsor $\widetilde{\pi} \vert U_i$ for $i  = 1,2,3,4$. 
\begin{Lemma} \label{Lemma:trivializationstauiTcirc}
Set $g = \begin{pmatrix} t & a_{-1}t^{-1}+a_0 \\ 0 & t^{-1}\end{pmatrix}$.
The section $O \rightarrow \widetilde{O}\;,\; g I \mapsto g f_i I^u$\;,\; of $\widetilde{\pi}$ extends
to a section $\sigma_i: U_i \rightarrow \widetilde{U_i}$, $i = 1,2,3,4$, of $\widetilde{\pi}$, where
\begin{align*}
f_1 = \begin{pmatrix} 1 & 0 \\ 0 & 1\end{pmatrix}\;\,\; f_2 = \begin{pmatrix} a_{-1} & 0 \\ 0 & \frac{1}{a_{-1}}\end{pmatrix}\; 
f_3 = \begin{pmatrix} a_0 & 0 \\ 0 & \frac{1}{a_0}\end{pmatrix}\; f_4 = \begin{pmatrix} \frac{a_0}{a_{-1}} & 0 \\ 0 & \frac{a_{-1}}{a_0}\end{pmatrix}\;.
\end{align*}
\end{Lemma}

\begin{proof}
For $i = 1$ the statement is clear. In the remaining cases the proof is similar to the proof of Lemma  \autoref{Lemma:Phi_i},
but now we instead multiply by suitable elements of $I^u$ from the right.
\end{proof}

The section $\sigma_i$ defines a trivialization $\tau_{i, T^{\circ}}: U_i \times T^{\circ} \xrightarrow{\cong} \widetilde{U_i}$ of
the $T^{\circ}$-torsor $\widetilde{\pi}\vert U_i$. 

\subsection{Local trivializations of $\sh{C} \vert \overline{O}$}
We first compute the transition functions w.r.t. local trivializations of $\sh{C}^{-1} \vert \overline{O}$, where $\sh{C}^{-1}$
denotes the line bundle inverse (dual) to $\mathcal{C}$. Let 
$\tau_{i, \Gm^{\central}}^{(-1)}: U_i \times \Aff^1 \xrightarrow{\cong} \sh{C}^{-1} \vert U_i$ be a trivialization of $\sh{C}^{-1} \vert U_i$. The corresponding transition
functions $t_{ij}^{(-1)}: U_i \cap U_j \rightarrow \Gm^{\central}$ are expressed through the $\tau_{i, \Gm^{\central}}^{(-1)}$ by
\begin{align*}
\tau_{i, \Gm^{\central}}^{(-1)-1} \circ \tau_{j, \Gm^{\central}}^{(-1)} : U_i \cap U_j \times \Aff^1 \xrightarrow{\cong}
U_i \cap U_j \times \Aff^1\;,\; (x,v) \mapsto \left(x, t_{ij}^{(-1)}(x)v\right)\;.
\end{align*}
Of course, the $t_{ij}^{(-1)}$ can be changed by a multiplicative constant. 

\begin{Lemma} \label{Lemma:transitionfunctionsforC-1}
We have
\begin{align*}
t_{12}^{(-1)} &= \frac{1}{x^3} \circ \Phi_1^{-1} \vert U_1 \cap U_2\\
t_{13}^{(-1)} &= \frac{1}{y} \circ \Phi_1^{-1} \vert U_1 \cap U_3\\
t_{14}^{(-1)} &= \frac{1}{xy} \circ \Phi_1^{-1} \vert U_1 \cap U_4 \;.
\end{align*}

\end{Lemma}

\begin{proof}
We have an isomorphism of line bundles $\sh{C}^{-1} \vert \overline{X_{s_1s_0}} \cong \Omega_{\overline{X_{s_1s_0}}}(\overline{X_{s_1}}+\overline{X_{s_0}})$, see e.g. \cite{Zhu10},
where the right-hand side are the top-degree forms on $\overline{X_{s_1s_0}}$ with possible simple poles along the Schubert divisors.
This description is the reason why we consider $\mathcal{C}^{-1}$ at all. Let $\omega_i \in \Gamma(U_i, \Omega_{\overline{X_{s_1s_0}}}(\overline{X_{s_1}}+\overline{X_{s_0}}))$ 
be a nowhere vanishing section (unique up to nonzero constant). We have $t_{ij}^{(-1)} = \frac{\omega_j}{\omega_i} \in \Gamma(U_i \cap U_j, \sh{O})^{\times}$. 
We have a canonical isomorphism
\begin{align*}
\Gamma(U_i, \Omega_{\overline{X_{s_1s_0}}}(\overline{X_{s_1}}+\overline{X_{s_0}})) = \begin{cases} \Gamma(U_1, \Omega_{U_1}) 
& i = 1 \\
\Gamma(U_2, \Omega_{U_2}(\Phi_2(\{0\}\times \Aff^1))) & i=2 \\
\Gamma(U_3, \Omega_{U_3}(\Phi_3(\Aff^1 \times \{0\}))) & i=3 \\
\Gamma(U_4, \Omega_{U_4}(\Phi_4(\Aff^1 \times \{0\} \cup \{0\} \times \Aff^1))) & i=4\;.
\end{cases}
\end{align*}
It is now easy to write down explicit expressions for the $\omega_i$, from which we then find the expressions for the $t_{ij}^{(-1)}$ given in the lemma. 
\end{proof}

The trivialization $\tau_{i, \Gm^{\central}}^{(-1)}$ induces a trivialization $\tau_{i, \Gm^{\central}}: U_i \times \Aff^1 \xrightarrow{\cong} \sh{C} \vert U_i  $ of $\sh{C} \vert U_i$ and the corresponding transition functions are $t_{ij} = \frac{1}{t_{ij}^{(-1)}}$. 

\subsection{Local coordinates on $\widetilde{\sh{C}}^{\times} \vert \widetilde{\overline{O}}$}
Combining the local coordinates and trivializations of the previous sections we now introduce local coordinates on the total
space of $\widetilde{\sh{C}}^{\times} \vert \widetilde{\overline{O}}$.
We define
\begin{align*}
\Phi_{i,T^{\circ},\Gm^{\central}} = \widetilde{\pi}^* \tau_{i, \Gm^{\central}} \circ (\tau_{i, T^{\circ}} \times \id) \circ (\Phi_i \times \id):
\Aff^2 \times T^{\circ} \times \Gm^{\central} \xrightarrow{\cong} \widetilde{C}^{\times} \vert \widetilde{U_i}
\end{align*}
 for $i=1,2,3,4$. 
This composition is $T^{\circ} \times \Gm^{\central}$-equivariant for the obvious $T^{\circ} \times \Gm^{\central}$-action on $\Aff^2 \times T^{\circ} \times \Gm^{\central}$ as all three composition factors are so.  
Using this we define
\begin{align*}
& \Psi_{ij, T^{\circ}, \Gm^{\central}} = \Phi_{i,T^{\circ}, \Gm^{\central}}^{-1} \circ \Phi_{j,T^{\circ}, \Gm^{\central}}: \\
& \Phi_j^{-1}(U_i \cap U_j) \times T^{\circ} \times \Gm^{\central} \xrightarrow{\cong}  \Phi_i^{-1}(U_i \cap U_j) \times T^{\circ} \times \Gm^{\central}\;.
\end{align*}

\begin{Cor}\label{Cor:Psiij} The isomorphism
\begin{align*}
\Psi_{12, T^{\circ}, \Gm^{\central}}: & \Gm \times \Aff^1 \times T^{\circ} \times \Gm^{\central} \xrightarrow{\cong} \Gm \times \Aff^1 \times T^{\circ} \times \Gm^{\central}\\
\Psi_{13, T^{\circ}, \Gm^{\central}}:& \Aff^1 \times \Gm \times T^{\circ} \times \Gm^{\central} \xrightarrow{\cong} \Aff^1 \times \Gm \times T^{\circ} \times \Gm^{\central} \\
\Psi_{14, T^{\circ}, \Gm^{\central}}:& \Gm^2 \times T^{\circ} \times \Gm^{\central} \xrightarrow{\cong} \Gm^2 \times T^{\circ} \times \Gm^{\central}
\end{align*}

is given by 

\begin{align*}
\left(x,y, \begin{pmatrix} a & 0 \\ 0 & a^{-1}\end{pmatrix}, v\right) &\mapsto \left(\frac{1}{x}, \frac{y}{x^2}, \begin{pmatrix} \frac{a}{x} & 0 \\ 0 & \frac{x}{a}\end{pmatrix}, \frac{1}{x^{3}} v\right)\\
& \mapsto \left(x, \frac{1}{y}, \begin{pmatrix} \frac{a}{y}  & 0 \\ 0 & \frac{y}{a}\end{pmatrix}, \frac{1}{y} v\right) \\
& \mapsto \left(\frac{1}{x},  \frac{1}{x^2 y}, \begin{pmatrix} \frac{a}{xy} & 0 \\
0 & \frac{xy}{a}\end{pmatrix}, \frac{1}{x^3 y} v\right)
\end{align*}

respectively. 

\end{Cor}
\begin{proof}
This is a direct computation using only Lemma \autoref{Lemma:Phi_i}, Lemma \autoref{Lemma:trivializationstauiTcirc},
and Lemma \autoref{Lemma:transitionfunctionsforC-1}. 
\end{proof}

\section{Local system $\sh{L}_{\Lambda, \kappa, \mu_{-1}, \mu_0}$}
The isomorphism $\Phi_{1,T^{\circ}, \Gm^{\central}}$ restricts to an isomorphism $\Gm^2 \times T^{\circ} \times \Gm^{\central} \xrightarrow{\cong} 
\widetilde{\sh{C}}^{\times} \vert \widetilde{O}$ denoted by the same symbol.
For $\Lambda, \kappa, \mu_{-1}, \mu_0 \in \C$ we define the local system $\sh{L}_{\Lambda,\kappa,\mu_{-1},\mu_0}$
of rank one in the category of right $\sh{D}$-modules on $\widetilde{\sh{C}}^{\times} \vert \widetilde{O}$ by
\begin{align*}
\Phi_{1, T^{\circ}, \Gm^{\central}}^* \sh{L}_{\Lambda, \kappa, \mu_{-1}, \mu_0} = 
\Omega_{\Lambda, \kappa, \mu_{-1}, \mu_0} = \Omega^{(\mu_{-1})}_{\Gm} \boxtimes \Omega^{(\mu_0)}_{\Gm}
\boxtimes \Omega^{(\Lambda)}_{T^{\circ}} \boxtimes \Omega^{(\kappa)}_{\Gm^{\central}}\;.
\end{align*}
Here we employed the rank-one local system $\Omega_A^{(\lambda)}$ in the category of right $\sh{D}$-modules on $A$
defined by $\Omega_A^{(\lambda)} = \sh{D}_A/(\xi_v+\lambda(v), v \in \mathfrak{a})\sh{D}_A$, where $A$ is any algebraic torus, $\mathfrak{a}$ is
its Lie algebra, $\xi_v$ is the translation vector field on $A$ given by $v \in \mathfrak{a}$, and $\lambda: \mathfrak{a} \rightarrow \C$ is 
a linear map. $\sh{D}_A$ is the sheaf of differential operators on $A$ and $(\xi_v+\lambda(v), v \in \mathfrak{a})\sh{D}_A$
is the right ideal of it generated by the indicated relations. In the case $A = \Gm$ we identify $\lambda$ canonically with a complex number. In the case of $\Omega^{(\Lambda)}_{T^{\circ}}$ we moreover 
use the isomorphism $\Gm \xrightarrow{\cong} T^{\circ}$ given by $a \mapsto \begin{pmatrix} a & 0 \\ 0 & a^{-1}\end{pmatrix}$ in order to be able to consider
$\Lambda$ as a complex number.  
Our choice of the parameter name $\mu_{-1}$ and $\mu_0$ 
is due to the fact that the corresponding coordinates in $O$, given by $\Phi_1$, are the coefficients of $t^{-1}$ and $t^0$. 
Of course, the local system $\sh{L}_{\Lambda,\kappa, \mu_{-1},\mu_0}$ can equivalently be viewed as a 
sheaf of modules on $O$ for the $(\Lambda,\kappa)$-twisted
differential operators on $O$, but in this article we only use this for terminological convenience in the title and some other places and not 
in computations. It is thus natural to refer to $\Lambda$ and $\kappa$ as the twist
parameters.

\section{Cleanness of the $\sh{D}$-module extension}
We recall the notion of cleanness of a $\sh{D}$-module extension \cite{Ber}. 
Let $\sh{M}$ be a holonomic right $\sh{D}$-module on a smooth variety $X$ and $\iota: X \hookrightarrow Y$ a locally closed affine embedding
into a smooth variety $Y$. There is a canonical morphism $\can_{\iota}: \iota_! \sh{M} \rightarrow \iota_* \sh{M}$ of holonomic right $\sh{D}$-modules on $Y$ that is the identity when restricted to $X$. The $\sh{D}$-module extension of $\sh{M}$ w.r.t. $\iota$ is called clean if $\can_{\iota}$ is an isomorphism. The following remark can be deduced from the construction of $\can_{\iota}$. 

\begin{Rem} \label{Rem:cleannnessislocal}
In case $\iota$ is an open affine embedding and $\{U_i\}_i$ is an open cover of $Y$ we have $\can_{\iota} \vert U_i = \can_{\iota_i}$ for each $i$,
where $\iota_i: \iota^{-1}(U_i) \hookrightarrow U_i$ is the restriction of $\iota$. Thus $\can_{\iota}$ is an isomorphism if and only if $\can_{\iota_i}$ is an isomorphism for each $i$. This means that the $\sh{D}$-module extension of $\sh{M}$ w.r.t. $\iota$ is clean if and only if the $\sh{D}$-module extension of 
the restriction $\sh{M} \vert \iota_i^{-1}(U_i)$ w.r.t. $\iota_i$ is clean for each $i$. 
\end{Rem}

\begin{Lemma} \label{Lemma:cleannessforOmegamu1boxtimesOmegamu2} Let $\mu_1, \mu_2 \in \C$.
The $\sh{D}$-module extension of $\Omega^{(\mu_1)}_{\Gm} \boxtimes \Omega^{(\mu_2)}_{\Gm}$
w.r.t. $\inc: \Gm^2 \hookrightarrow \Aff^2$ is clean if and only if $\mu_1 \notin \Z$ and $\mu_2 \notin \Z$. 
\end{Lemma}

\begin{proof}
We first argue that the extension of $\Omega^{(\epsilon \mu_1)}_{\Gm} \boxtimes \Omega^{(\epsilon \mu_2)}_{\Gm}$ w.r.t. $\inc$ is clean
for $\epsilon \in \{\pm 1\}$ if and only if $\inc_{\cdot}(\Omega^{(\epsilon \mu_1)}_{\Gm} \boxtimes \Omega^{(\epsilon \mu_2)}_{\Gm})$ is
a simple $\sh{D}$-module on $\Aff^2$ for $\epsilon \in \{\pm 1\}$. We know \cite{Ber} that the image of $\can_{\inc}: \inc_!(\Omega^{(\mu_1)}_{\Gm} \boxtimes \Omega^{(\mu_2)}_{\Gm})
\rightarrow \inc_{\cdot}(\Omega^{(\mu_1)}_{\Gm} \boxtimes \Omega^{(\mu_2)}_{\Gm})$
is a simple $\sh{D}$-module on $\Aff^2$ and thus obtain one implication. For the other implication we use that if $\inc_{\cdot}(\Omega^{(\mu_1)}_{\Gm} \boxtimes \Omega^{(\mu_2)}_{\Gm})$ is simple, then $\can_{\inc}$ surjects and hence
$\Dual \can_{\inc}$, which can be identified with $\can_{\inc}: \inc_!(\Omega^{(-\mu_1)}_{\Gm} \boxtimes \Omega^{(-\mu_2)}_{\Gm})
\rightarrow \inc_{\cdot}(\Omega^{(-\mu_1)}_{\Gm} \boxtimes \Omega^{(-\mu_2)}_{\Gm})$, injects. Here $\Dual$ denotes the 
holonomic duality functor on $\Aff^2$.
Finally, it is easy to see that $\inc_{\cdot}(\Omega^{(\mu_1)}_{\Gm} \boxtimes \Omega^{(\mu_2)}_{\Gm})$ is a simple $\sh{D}$-module on $\Aff^2$
if and only if $\mu_1 \notin \Z$ and $\mu_2 \notin \Z$. Indeed, this follows from elementary arguments using the explicit action of $x,y,\partial_x, \partial_y$ on the elements of the natural $\C$-basis of $\Gamma(\Gm, \Omega_{\Gm}^{(\mu_1)}) \otimes_{\C} \Gamma(\Gm, \Omega^{(\mu_2)}_{\Gm})$.
\end{proof}

The following theorem is the main result of this work. 
\begin{Thm}\label{Thm:cleanness}
Let $\Lambda, \kappa, \mu_{-1}, \mu_0 \in \C$. 
The $\sh{D}$-module extension of $\sh{L}_{\Lambda, \kappa, \mu_{-1}, \mu_0}$ w.r.t.  the inclusion $\inc: \widetilde{\sh{C}}^{\times} \vert \widetilde{O}
\hookrightarrow \widetilde{\sh{C}}^{\times}$ is clean if and only if $\mu_{-1} \notin \Z$
and $\mu_0 \notin \Z$ and $\mu_{-1}+2\mu_0+\Lambda + 3\kappa \notin \Z$
and $\mu_0+\Lambda+\kappa \notin \Z$. \end{Thm}

\begin{proof}
By Remark \autoref{Rem:cleannnessislocal}, as $\{\widetilde{U}_i\}_{i=1,2,3,4}$ is an open cover of $\widetilde{\overline{O}}$, the extension of $\sh{L}_{\Lambda, \kappa, \mu_{-1}, \mu_0}$ w.r.t. the locally closed affine inclusion $\widetilde{\sh{C}}^{\times} \vert \widetilde{O} \hookrightarrow \widetilde{\sh{C}}^{\times}$
is clean if and only if it is clean w.r.t. the open affine inclusion $\widetilde{\sh{C}}^{\times} \vert \widetilde{O} \hookrightarrow \widetilde{\sh{C}}^{\times} \vert \widetilde{U_i}$
for each $i=1,2,3,4$. 
We have a commutative diagram
\begin{align*}
\xymatrix{ \widetilde{\sh{C}}^{\times} \vert \widetilde{O} \ar[rr]^{\inc} && \widetilde{\sh{C}}^{\times} \vert \widetilde{U_i} \\
\Gm^2\times T^{\circ} \times \Gm^{\central} \ar[u]_{\cong}^{\Phi_{i, T^{\circ}, \Gm^{\central}}} \ar[rr]_{\inc \times \id_{T^{\circ} \times \Gm^{\central}}}
&& \Aff^2 \times T^{\circ} \times \Gm^{\central} \ar[u]^{\cong}_{\Phi_{i, T^{\circ}, \Gm^{\central}}}}
\end{align*}
for each $i=1,2,3,4$. 
Hence, from the definition of $\sh{L}_{\Lambda,\kappa,\mu_{-1},\mu_0}$  and
 Lemma \autoref{Lemma:cleannessforOmegamu1boxtimesOmegamu2}
it follows that for $i=1$ this is the case if and only if $\mu_{-1} \notin \Z$ and $\mu_0 \notin \Z$. 
From the definition of $\sh{L}_{\Lambda,\kappa,\mu_{-1},\mu_0}$ and Corollary \autoref{Cor:Psiij} we conclude

\begin{align*}
\Phi^*_{i,T^{\circ}, \Gm^{\central}} \sh{L}_{\Lambda,\kappa,\mu_{-1},\mu_0} &= \Psi^*_{1i,T^{\circ}, \Gm^{\central}} \Omega_{\Lambda, \kappa, \mu_{-1}, \mu_0} = 
\begin{cases} 
\Omega_{\Lambda, \kappa, -\mu_{-1}-2\mu_0-\Lambda-3\kappa, \mu_0} & i=2 \\
\Omega_{\Lambda, \kappa, \mu_{-1}, -\mu_0-\Lambda-\kappa} & i=3 \\
\Omega_{\Lambda, \kappa, -\mu_{-1}-2\mu_0-\Lambda-3\kappa, -\mu_0-\Lambda-\kappa} & i=4 \;.
\end{cases}
\end{align*}
We used the canonical isomorphism $m^* \Omega^{(\mu)}_{\Gm} =\Omega^{(\mu)}_{\Gm} \boxtimes \Omega^{(\mu)}_{\Gm}$
 for the multiplication morphism $m: \Gm \times \Gm \rightarrow \Gm$. 
According to Lemma \autoref{Lemma:cleannessforOmegamu1boxtimesOmegamu2}, for $i=2$ the extension is thus clean if and only if $\mu_{-1} + 2\mu_0+\Lambda+3\kappa \notin \Z$
and $\mu_0 \notin \Z$. 
For $i=3$ it is clean if and only if $\mu_{-1} \notin \Z$ and $\mu_0+\Lambda+\kappa \notin \Z$. 
For $i=4$ it is clean if and only if $\mu_{-1}+2\mu_0+\Lambda+3\kappa \notin \Z$ and $\mu_0+\Lambda+\kappa \notin \Z$. 
\end{proof}

Each of the four non-integrality conditions of the theorem corresponds to an irreducible component of $\overline{O} \setminus O$. 
Note that in the case where the twist parameters are integral, i.e. $\Lambda, \kappa \in \Z$, we can view $\sh{L}_{\Lambda,\kappa,
\mu_{-1},\mu_0}$ as a local system in the category of right $\sh{D}$-modules on $O$, and the non-integrality conditions of the theorem reduce to $\mu_{-1}, \mu_0, \mu_{-1}+2\mu_0 \notin \Z$.
\begin{Rem}\label{Rem:cleannessequivalenttosimplicity} In the theorem, cleanness can equivalently be replaced by simplicity of $\inc_{\cdot} \sh{L}_{\Lambda, \kappa, \mu_{-1}, \mu_0}$
as a right $\sh{D}$-module on $\widetilde{C}^{\times} \vert \widetilde{\overline{O}}$ and also of
$\inc_{*}\sh{L}_{\Lambda, \kappa, \mu_{-1}, \mu_0}$ as a right $\sh{D}$-module on $\widetilde{C}^{\times}$.
\end{Rem}
It would  be interesting to describe the subquotients of the $!$- and $*$-extension of $\sh{L}_{\Lambda, \kappa, \mu_{-1}, \mu_0}$ in case at least one of the non-integrality conditions of the theorem is violated. \par

Let us finally comment on the implication of this result for the global sections of the right $\sh{D}$-module $\inc_* \sh{L}_{\Lambda, \kappa, \mu_{-1}, \mu_0}$
 on $\widetilde{C}^{\times}$. Because we can identify $\widetilde{C}^{\times}$ with a version of
 the enhanced affine flag variety of $\SL_2$, the correct global sections functor $\Gamma$ is defined in \cite{BD00}.
$\Gamma(\inc_* \sh{L}_{\Lambda, \kappa, \mu_{-1}, \mu_0})$ thus becomes a representation of the affine Kac-Moody algebra $\widehat{\mathfrak{sl}_2}$
of a certain level. It would be interesting to describe $\Gamma(\inc_* \sh{L}_{\Lambda, \kappa, \mu_{-1}, \mu_0})$ by an explicit algebraic construction. 
Let us assume that $(\Lambda, \kappa)$ is regular and antidominant in the
sense of \cite{BD00}[7.15.5]. If the hope of \cite{BD00}[7.15.7 (ii)] is true,
then, in view of \cite{BD00}[7.15.6 Theorem] and Remark \autoref{Rem:cleannessequivalenttosimplicity},
$\Gamma(\inc_* \sh{L}_{\Lambda, \kappa, \mu_{-1}, \mu_0})$ is an irreducible $\widehat{\mathfrak{sl}_2}$-representation if and
only if the condition on $\Lambda, \kappa, \mu_{-1}, \mu_0$ of Theorem \autoref{Thm:cleanness} holds. Even if the hope is false,
this condition is still necessary for the $\widehat{\mathfrak{sl}_2}$-representation to be irreducible.

\bibliographystyle{alpha}
\bibliography{references}

\begin{thebibliography}{Zhu10}

\bibitem[BD]{BD00}
A.~Beilinson and V.~Drinfeld.
\newblock {\em Quantization of {Hitchin}'s integrable system and {Hecke}
  eigensheaves}.

\bibitem[Ber]{Ber}
J.~Bernstein.
\newblock Algebraic theory of {D}-modules.
\newblock available at
  \href{http://www.math.uchicago.edu/~mitya/langlands/Bernstein/Bernstein-dmod.ps}{\nolinkurl{http://www.math.uchicago.edu/~mitya/langlands.html}}.
\newblock unpublished.

\bibitem[Eic16]{Eic16b}
C.~Eicher.
\newblock Relaxed highest weight modules from $\mathcal{D}$-modules on the
  {Kashiwara} flag scheme.
\newblock arXiv:1607.06342 [math.RT], 2016.
\newblock submitted to Advances in Mathematics.

\bibitem[Eic20]{Eic20}
C.~Eicher.
\newblock Localization of relaxed {Verma} modules over affine {Kac-Moody}
  algebras: the case of not necessarily integral parameters, 2020.
\newblock in preparation.

\bibitem[Zhu10]{Zhu10}
X.~Zhu.
\newblock Loop groups and their flag varieties, 2010.
\newblock Lecture notes.

\end{thebibliography}
\end{document}